\documentclass[10pt,smallextended,envcountsect,final]{svjour3}

\usepackage{color}

\usepackage[UKenglish]{babel}
\usepackage{stmaryrd}
\usepackage{hyperref}

\usepackage{latexsym, amssymb, amsfonts, enumerate}

\usepackage{amsmath} 
\usepackage{amsthm}
\usepackage{amsopn}
\usepackage{mathrsfs}

\DeclareMathSymbol{\R}{\mathalpha}{AMSb}{"52}

\newcommand{\N}{\mathbb{N}}

\newcommand{\F}{\mathcal{F}}
\newcommand{\cF}{\mathcal{F}}
\newcommand{\cB}{\mathcal{B}}
\newcommand{\E}{\mathbb{E}}
\renewcommand{\P}{\mathbb{P}}
\newcommand{\bR}{\mathbb{R}}

\usepackage{dsfont}

\theoremstyle{plain}

\numberwithin{equation}{section}

\begin{document}

\author{Istv\'an Gy\"ongy, David \v{S}i\v{s}ka}
\authorrunning{I.~Gy\"ongy, D.~\v{S}i\v{s}ka}
\institute{
Istv\'an Gy\"ongy 
\at School of Mathematics, University of Edinburgh\\
\email{i.gyongy@ed.ac.uk}
\and
David \v{S}i\v{s}ka
\at School of Mathematics, University of Edinburgh\\
\email{d.siska@ed.ac.uk}% Tel. +44 131 651 9091, Fax. +44 131 650 6553
}

\titlerunning{It\^o Formula in Intersection of Banach Spaces}
\title{It\^o Formula for Processes Taking Values in Intersection of Finitely Many Banach Spaces\thanks{This an electronic reprint of the article published in {\em Stochastics and Partial Differential Equations: Analysis and Computations} available
at \url{http://dx.doi.org/10.1007/s40072-017-0093-6}.
It may differ from the published version in typographical detail.}
}

\date{Submitted for publication: 6th September 2016, First Online: 17th March 2017.\\
\textcircled{c} The Author(s) 2017. This article is an open access publication.}

\maketitle

\begin{abstract}
Motivated by applications to SPDEs we extend the It\^o formula 
for the square of the norm of a semimartingale $y(t)$
from \cite{gyongy:krylov:on:stochastic:II}
to the case 
\begin{equation*}
\sum_{i=1}^m \int_{(0,t]} v_i^{\ast}(s)\,dA(s) + h(t)=:y(t)\in V
\quad \text{$dA\times \P$-a.e.}, 
\end{equation*}
where  $A$ is an 
increasing right-continuous adapted process, $v_i^{\ast}$ is a 
progressively measurable process with values in $V_i^{\ast}$, the dual of a Banach space $V_i$, $h$ is a cadlag 
martingale with values in a Hilbert space $H$, identified with its dual $H^{\ast}$,  
and $V:=V_1\cap  V_2 \cap \ldots \cap V_m$  is continuously and densely 
embedded in $H$.  

The formula is proved under the condition that 
$\|y\|_{V_i}^{p_i}$ and $\|v_i^\ast\|_{V_i^\ast}^{q_i}$ are 
almost surely locally integrable with respect to $dA$ for some conjugate
exponents $p_i, q_i$. 
This condition is essentially weaker than the one which would arise 
in application of the results in~\cite{gyongy:krylov:on:stochastic:II}
to the semimartingale above. 
 \end{abstract}


\section{Introduction}
It\^o formula for the square of the norm is an essential tool 
in the study of stochastic evolution equations 
of the type 
\begin{equation}                                                      \label{evolution}
dv(t)=\mathbb A(t,v(t))\,dt+\sum_k\mathbb B_k(t,v(t))\,dW^k(t),
\end{equation}
where $(W^k)_{k=1}^{\infty}$ is a sequence of independent 
Wiener processes, and $\mathbb A(t,\cdot)$ and $\mathbb B_k(t,\cdot)$ 
are  (possibly random nonlinear) operators on a separable 
real Banach space $V$, with values  
in a Banach space $V'$ and a Hilbert space $H$ respectively,
such that $V\hookrightarrow H\hookrightarrow V'$ with continuous and 
dense embeddings.  
We assume there is a constant $K$ 
such that $(v,h)\leq K\|v\|_V\|h\|_{V'}$ for all $v\in V$ 
and $h\in H$. This means that for the linear mapping $\Psi:H\to H^{\ast}$, which identifies 
$H$ with its dual $H^{\ast}$ via the inner product in $H$, 
we have $\|\Psi(h)\|_{V^{\ast}}\leq K\|h\|_{V'}$. 
Therefore, since $H$ is dense in $V'$, $\Psi$ 
can be extended to a continuous mapping from $V'$ into $V^{\ast}$, 
the dual of $V$. 
 It is assumed that this extension 
is one-to-one from $V'$ into $V^{\ast}$. 
Thus an initial value problem for 
equation \eqref{evolution}
can be viewed 
as 
\begin{equation}                                    \label{y}
v(t)=\int_0^tv^{\ast}(s)\,ds+h(t)=:y(t)
\end{equation}
with the $V^{\ast}$-valued process $v^{\ast}(t):=\mathbb A(t,v(t))$ and 
$H\equiv H^{\ast}$-valued process 
$$
h(t):=h_0+\sum_k\int_0^t\mathbb B^k(s,v(s))\,dW^k(s),
$$
where $h_0$ is a given initial value and the equality \eqref{y} 
in $V^{\ast}$  
is required $ dt\times\P$ 
almost everywhere. 
In the special case  $B_k=0$ for every $k$, and nonrandom $h_0$ and $A$,  i.e., 
in the case 
$$
v(t)=h_0+\int_0^tv^{\ast}(s)\,ds, \quad dt\text{-a.e.},
$$
it is well-known  that when $v\in L_p([0,T], V)$,  
$v^{\ast}\in L_q([0,T], V^{\ast})$ for $T>0$ and conjugate exponents 
$p$ and $q$, then there is $u\in C([0,T],H)$  such that $u=v$ for $dt$-almost 
all $t\in[0,T]$ and the ``energy equality" 
$$
|u(t)|_H^2=|h_0|_H^2+2\int_0^t\langle v^{\ast}(s),v(s)\rangle\,ds
$$
holds for all $t\in[0,T]$, where $\langle\cdot,\cdot\rangle$ denotes the 
duality pairing of $V^{\ast}$ and $V$. 
This formula is used in proofs of existence and uniqueness theorems for PDEs, 
see e.g., \cite{Evans} and \cite{Lions}. A generalisation of it, 
a ``stochastic energy equality", i.e., 
an It\^o formula for the square of the $H$-norm of $y$, was first presented in Pardoux \cite{pardoux:thesis},   
and was used to obtain existence and uniqueness theorems for SPDEs.  
The proof of it in \cite{pardoux:thesis} was not separated 
from the theory of SPDEs developed there. A proof, not bound
to the theory of SPDEs,   
was given in Krylov and Rozovskii \cite{krylov:rozovskii:stochastic},   
and then this stochastic energy equality 
was generalised in Gy\"ongy and Krylov~\cite{gyongy:krylov:on:stochastic:II} 
to $V^{\ast}$-valued semimartingales $y$ 
of the form 
\begin{equation}
\label{eq y}
y(t)=\int_{(0,t]}v^{\ast}(s)\,dA(s)+h(t), 
\end{equation}
where $A$ is an adapted nondecreasing cadlag process and 
$h$ is an $H$-valued cadlag martingale. This generalisation is used 
in Gy\"ongy~\cite{gyongy} to extend the theory of SPDEs developed in \cite{pardoux:thesis} and 
\cite{krylov:rozovskii:stochastic} 
to SPDEs driven by random orthogonal measures and L\'evy martingales,  
written in the form 
\begin{equation}                                                 \label{SPDE}
dv(t)=\mathbb A(t,v(t))\,dA(t)+\mathbb B(t,v(t))\,dM(t)
\end{equation}
with cadlag (quasi left-continuous) martingales $M$ with values in a Hilbert space. 

In the present paper we are interested in stochastic energy equalities 
which can be applied to SPDEs \eqref{SPDE}  when 
$\mathbb A$ is of the form $\mathbb A=\mathbb A_1+\mathbb A_2+\cdots+\mathbb A_m$ 
and the operators $\mathbb A_i$ have different analytic and growth properties.   
This means, 
$$
\mathbb A_i(t,\cdot):V_i\to V_i^{\prime}\quad i=1,2,\ldots,m
$$ 
for some Banach spaces $V_i$ and $V'_i$, 
such that with a constant $R$ and a process $g$, locally integrable
with respect to $dA$, one has for all $t$
$$
\|\mathbb A_i(t,w)\|_{V'_i} \leq |g_t|^{1/q_i} + R\|w\|_{V_i}^{p_i-1}
$$
for all $w\in V$, $q_i=p_i/(p_i-1)$ with (possibly) different exponents $p_i\geq1$, 
which for $p_i=1$ means that $\|\mathbb A_i(t,w)\|_{V'_i}$ is bounded by a constant. 
 
In the special case when $A(t)=t$ and $M$ is a Wiener process 
the above situation was considered in \cite{pardoux:thesis}, and a related stochastic energy 
equality was also presented there. 
Our main result, Theorem \ref{thm:1} generalises the results on stochastic energy equalities 
from \cite{pardoux:thesis} and \cite{gyongy:krylov:on:stochastic:II}. We prove it by adapting 
the method of the proof of the main theorem in \cite{gyongy:krylov:on:stochastic:II}. 

In the present paper we consider a semimartingale $y$ of the form~\eqref{eq y} 
such that $dA \times \P$-almost everywhere $y$ takes values in 
$V=V_1\cap \ldots \cap V_m$, where $V_i$ are Banach spaces (over $\mathbb{R}$) 
such that $V$ with the norm $\|\cdot\| := \sum_{i=1}^m \|\cdot\|_{V_i}$
is continuously and densely embedded in $H$.
The process $v^\ast$ in~\eqref{eq y} is of the form $v^\ast = \sum_{i=1}^m v_i^\ast$,
where $v_i^\ast$ are $V_i^\ast$-valued progressively measurable processes.
We prove that $y$ is almost surely cadlag as a process with values in $H$ 
and for $|y|_H^2$ an It\^o formula holds under the assumption that
$\|y\|_{V_i}^{p_i}$ and $\|v_i^\ast\|_{V_i^\ast}^{q_i}$ are 
almost surely locally integrable with respect to $dA$ for some conjugate
exponents $p_i, q_i$. 
See Section~\ref{sec:main} for precise formulation of the main theorem.
To apply the result of~\cite{gyongy:krylov:on:stochastic:II} to $y$ 
given by~\eqref{eq y}, one needs the local integrability (with respect to $dA$) 
of 
\[
\|y\|_V\|v^\ast\|_{V^\ast} = \left(\|y\|_{V_1} + \cdots + \|y\|_{V_m} \right)\|v_1^\ast + \cdots + v_m^\ast\|_{V^\ast},
\]
which, in general, is not satisfied under our assumptions.
See Remark~\ref{rem integrability} and Example~\ref{ex spde}. 

We note that in the context of stochastic evolution 
equations it is possible to prove It\^o formulae for more general functions
(satisfying appropriate differentiability assumptions), 
see again Pardoux~\cite{pardoux:thesis}, 
Krylov~\cite{krylov:ito_fla}, \cite{K2010}, \cite{K},
Da Prato, Jentzen and R\"ockner~\cite{daprato:jentzen:rockner},
as well as Dareiotis and Gy\"ongy~\cite{dareiotis:gyongy}.
The It\^o formula for the square of the norm is 
used in particular to establish a priori estimates as well as 
uniqueness and existence of solutions of stochastic evolution equations.
The more general It\^o formula can then be used to study finer properties 
of solutions of stochastic evolution equations,
for example the maximum principle.  

For general theory of SPDEs in the variational setting we refer the reader 
to Krylov and Rozovskii~\cite{krylov:rozovskii:stochastic},
Pr\'ev\^ot and R\"ockner~\cite{PR} and Rozovskii~\cite{R}.
\section{Main Results}
\label{sec:main}

For $i=1,\ldots,m$ let $(V_i,\|\cdot\|_{V_i})$ be 
real Banach 
spaces with duals $(V_i^*,\|\cdot\|_{V_i^*})$.  
Let $V$ denote the vector space $V_1\cap \cdots \cap V_m$ 
with the norm $\|\cdot\| := \|\cdot\|_{V_1} + \cdots + \|\cdot\|_{V_m}$. 
Then clearly, $V$ is a Banach space. Assume that it is separable and 
is continuously and densely 
embedded in a Hilbert space $(H, |\cdot|)$, which is identified with its dual $H^{\ast}$ 
by the help of the inner product $(\cdot,\cdot)$ in $H$. Thus we have 
$$
V\hookrightarrow H\equiv H^{\ast}\hookrightarrow V^{\ast}, 
$$
where $H^{\ast}\hookrightarrow V^{\ast}$ is the adjoint of the embedding 
$V\hookrightarrow H$. 
We use the notation 
$\langle\cdot, \cdot\rangle$ for the duality pairing between $V$ and $V^{\ast}$. 
Note that if $v^{\ast}\in V_i^{\ast}$ for some $i$, then its restriction to $V$ belongs 
to $V^{\ast}$ and $|\langle v^{\ast},v\rangle|\leq \|v^{\ast}\|_{V^{\ast}_i}\|v\|_{V_i}$ 
for all $v\in V$. Note also that $\langle v^{\ast},v\rangle=(h,v)$ for  
for all $v\in V$ when $v^{\ast}=h\in H$. 

A complete probability space $(\Omega, \F, \P)$ 
together with an increasing family of $\sigma$-algebras 
$(\F_t)_{t\geq 0}$, $\F_t \subset \F$ will be used throughout the paper.
Moreover it is assumed that the usual conditions are satisfied: 
$\bigcap_{s> t} \F_s = \F_t$ and $\F_0$ contains all subsets of $\P$-null sets of $\F$. 
We use the notation $\cB(\bR_+)$ for the $\sigma$-algebra of Borel subsets of 
$\bR_+=[0,\infty)$, and for a real-valued  increasing $\cB(\bR_+)\otimes\cF$-measurable 
process $(A(t))_{t\geq0}$ the notation $dA\times\P$ stands for the measure  
defined on $\cB(\bR_+)\otimes\cF$ 
by
$$
(dA\times\P)(F)=\E\int_0^{\infty}{\bf 1}_F\,dA(t), \quad F\in \cB(\bR_+)\otimes\cF. 
$$

Let $h=(h(t))_{t\geq0}$ be an $H$-valued locally square integrable martingale that is cadlag 
(continuous from the right with left-hand limits) in the strong 
topology on $H$. 
Its quadratic variation process is denoted by  $[h]$, and $\langle h\rangle$ denotes 
the unique predictable process starting from zero such that $|h|^2-\langle h\rangle$ is a 
local martingale. 
Furthermore let $A$ be a real-valued nondecreasing adapted 
cadlag process starting from zero.
Finally let $v=(v(t))_{t\geq0}$ be a $V$-valued progressively measurable process 
and for $i=1,\ldots,m$ let $v^{\ast}_i=(v^{\ast}_i)_{t\geq0}$ be $V_i^*$-valued processes 
such that $\langle\varphi,v_i^{\ast}\rangle$ are progressively measurable for any
$\varphi \in V$. Notice that $v$ is also progressively measurable as a process 
with values in $\bar V_i$, the closure in $V_i$-norm of the linear hull of  
$
\{v(t):t\geq0, \omega\in\Omega\}.
$

Let there be $p_i \in [1,\infty)$ and $q_i=p_i/(p_i-1)\in(1,\infty]$, 
where, as usual, $1/0:=\infty$.
Assume that for each $i=1,2,\ldots,m$ and $T>0$ 
\begin{equation}                                                                \label{assumption main}
\int_0^T\|v(t)\|_{V_i}^{p_i}\,dA(t)<\infty,\quad 
\left(\int_0^T\eta_i^{q_i}(t)\,dA(t)\right)^{1/q_i}<\infty, 
\end{equation}
for some progressively measurable process $\eta_i$ such that 
$\|v_i^{\ast}\|_{V_i^{\ast}}\leq \eta_i$ for $dA\times \P$-almost everywhere, 
where  for $q_i=\infty$  the second expression means  
$$
\text{$dA$-ess}\,\sup_{t\leq T}\eta_i(t), 
$$
the essential supremum (with respect to $dA$) of $\eta_i$ over $[0,T]$. 

The following theorem is the main result of this paper.

\begin{theorem}                                                                \label{thm:1}
Let $\tau$ be a stopping time. 
Suppose that for all $\varphi \in V$ and for $dA\times \P$ 
almost all $(\omega, t)$ such that $t\in (0,\tau(\omega))$ we have
\begin{equation}
\label{eq:1}
(v(t),\varphi)= \sum_{i=1}^m \int_{(0,t]}\langle v^{\ast}_i(s),\varphi\rangle\, dA(s) 
+ (h(t),\varphi).
\end{equation}
Then there is $\tilde{\Omega} \subset \Omega$ 
with $\P(\tilde{\Omega}) = 1$ and
an $H$-valued cadlag process $\tilde v$ such that 
the following statements hold.
\begin{enumerate}
\item[(i)] For $dA\times \P$ almost all $(t, \omega)$ 
satisfying $t \in (0,\tau(\omega))$ we have $\tilde v=v$.
\item[(ii)] For all $\omega \in \tilde{\Omega}$ and $t\in[0,\tau(\omega))$ we have
\begin{equation}
                                                                                                                 \label{eq:2}
(\tilde v(t),\varphi)
 = \sum_{i=1}^m \int_{(0,t]}\langle v^{\ast}_i(s),\varphi\rangle\,dA(s) 
 + h(t)\varphi\quad \textrm{for all }\, \varphi \in V.
\end{equation}
\item[(iii)] For all $\omega \in \tilde{\Omega}$ and $t\in[0, \tau(\omega))$ 
\begin{equation}
\label{eq:3}
\begin{split}
|\tilde v(t)|^2  = 
& |h(0)|^2 + 2\sum_{i=1}^m \int_{(0,t]}\langle v^{\ast}_i(s),v(s)\rangle\,dA(s) 
+ 2\int_{(0,t]}( \tilde v(s-)\,dh(s)) \\ 
& - \int_{(0,t]} |v^{\ast}(s)|^2 \Delta A(s) dA(s) + [h]_t, 
\end{split}
\end{equation}
where $v^{\ast}(t):=\sum_{i=1}^m v^{\ast}_i(t)\in H$ 
for $\Delta A(t) > 0$. 
\end{enumerate}
\end{theorem}

Consider now a situation where the assumptions on 
$h$ and $A$ are as above but $m=1$ and regarding $v$ and $v^{\ast}:=v^{\ast}_1$
we know that $\|v(t)\|$, $\|v^{\ast}(t)\|_{V^*}$ and $\|v(t)\|\|v^{\ast}(t)\|_{V^*}$
are almost surely locally integrable with respect to $dA(t)$.
Let 
\begin{equation*}
\bar{v}^{\ast}(t) := \frac{v^{\ast}(t)}{1+\|v^{\ast}(t)\|_{V^*}}
\,\,\, \textrm{and} \,\,\,
\bar A(t) := \int_{(0,t]}(1+\|v^{\ast}(t)\|_{V^*})dA(t).	
\end{equation*}
Then $\|\bar{v}^{\ast}\|_{V^*} \leq 1$ and so $v$, $\bar{v}^{\ast}$ 
and $\bar A$ satisfy the conditions 
on $v$, $v^{\ast}$ and $A$, respectively, with $p_1 = 1$
and $q_1 = \infty$. 
If~\eqref{eq:1} holds 
for all $\varphi \in V$ and for $dA\times \P$ almost all $(\omega, t)$ 
such that $t\in (0,\tau(\omega))$ 
then
\begin{equation*}
( v(t), \varphi ) 
= \sum_{i=1}^m \int_{(0,t]} \langle \bar{v}^{\ast}(s),\varphi \rangle\, d\bar A(s) 
+ ( h(t),\varphi ).
\end{equation*}
Applying Theorem~\ref{thm:1} then means that we have all of its 
conclusions with $\bar{v}^{\ast}$ and $\bar A$ in place of $v^{\ast}$ and $A$ respectively.
In particular, we get 
\begin{equation*}
                                                                                                        \begin{split}
|\tilde v(t)|^2  = 
& |h(0)|^2 + 2 \int_{(0,t]} \langle \bar{v}^{\ast}(s), v(s) \rangle\, d\bar A(s) 
+ 2\int_{(0,t]} (\tilde v(s-),dh(s)) \\ 
& - \int_{(0,t]} \left|\bar{v}^{\ast}(s)\right|^2 \Delta \bar A(s) d\bar A(s) + [h]_t\\
= & |h(0)|^2 + 2 \int_{(0,t]} \langle v^{\ast}(s), v(s) \rangle\, dA(s) 
+ 2\int_{(0,t]} (\tilde v(s-),dh(s)) \\ 
& - \int_{(0,t]} \left|v^{\ast}(s)\right|^2 \Delta A(s)\,dA(s) + [h]_t.
\end{split}
\end{equation*}
Hence we see that Theorem~\ref{thm:1} is a generalisation of 
the main theorem in Gy\"ongy and 
Krylov~\cite{gyongy:krylov:on:stochastic:II}.

\begin{remark}
\label{rem integrability}
One might think that Theorem \ref{thm:1} 
follows from the main theorem 
in~\cite{gyongy:krylov:on:stochastic:II} by considering the process 
$v^{\ast}=\sum_iv^{\ast}_i$ as a process with values in $V^{\ast}$. 
However,  taking into account  
that for any $w^{\ast}\in V^{\ast}$ 
$$
\|w^{\ast}\|_{V^{\ast}}
=\inf\Big\{\max_{i=1,\ldots,m}\|w_i^{\ast}\|_{V^{\ast}_i}
:w^{\ast}=\sum_{i=1}^mw_i^{\ast}, w_i^{\ast}\in V_i^{\ast}\Big\}
$$
(see for example Gajewski, 
Gr{\"o}ger and Zacharias~\cite[Chapter 1, Theorem 5.13]{ggz}), 
one can show that the local integrability condition in \cite{gyongy:krylov:on:stochastic:II}
for 
$$
\|v\|_{V}\|v^{\ast}\|_{V^{\ast}}=(\|v\|_1+\cdots+\|v\|_m)\|v^{\ast}\|_{V^{\ast}}
$$ 
is not implied by our assumption \eqref{assumption main}. 
Thus the main theorem in  \cite{gyongy:krylov:on:stochastic:II} 
is not applicable in our situation. 
\end{remark}

We consider the following motivating example.
\begin{example}                                                                                         \label{ex spde}
Consider the stochastic partial differential equation
\begin{equation*}
\begin{split}
du = & \left[ \nabla(|\nabla u|^{p_1-2} \nabla u) 
+ |u|^{p_2-2} u \right] dt\\ 
& + f(u,\nabla u)\,dW 
+ \int_Z g(u)\,q(dt,dz)\,\, \textrm{in} \,\, \mathscr{D}\times (0,T).
\end{split}
\end{equation*}
Here $W$ is a Wiener process (finite or infinite dimensional depending
on the choice of $f$), 
$(Z,\Sigma)$ is a measurable space and $q(ds,dz)$ 
a stochastic martingale measure on $[0,\infty)\times Z$.
See, for example, Gy\"ongy and Krylov~\cite{gyongy:krylov:on:stochastic:I}
for detailed definition.
We take $\mathscr{D}$ to be a bounded Lipschitz domain in $\R^d$.

It is natural to assume that a solution $u$ should be such that 
$\|u\|_{W^{1}_{p_1}(\mathscr{D})}^{p_1}$
and
$\|u\|_{L_{p_2}(\mathscr{D})}^{p_2}$ are almost surely
locally integrable.  
To apply the result in 
Gy\"ongy and Krylov~\cite{gyongy:krylov:on:stochastic:II} 
one could try to take
$V := W^{1}_{p_1}(\mathscr{D}) \cap L_{p_2}(\mathscr{D})$
with the norm 
$\|\cdot\|_V = \|\cdot\|_{W^{1}_{p_1}(\mathscr{D})} +\|\cdot\|_{L_{p_2}(\mathscr{D})}$.
The dual of $V$ can be identified with the linear space
\begin{equation*}
V^* = \{f = f_1 + f_2 :f_1 \in W^{1}_{p_1}(\mathscr{D})^*, 
f_2 \in L_{p_2}(\mathscr{D})^*\}	
\end{equation*}
equipped with the norm
\begin{equation*}
\begin{split}
\|f\|_{V*} = \inf \{\max(&\|f_1\|_{W^{1}_{p_1}(\mathscr{D})^*},
\|f_2\|_{ L_{p_2}(\mathscr{D})^*}): \\
&  f = f_1 + f_2, f_1 \in W^{1}_{p_1}(\mathscr{D})^*,\,\, f_2\in L_{p_2}(\mathscr{D})^* \}.	
\end{split}
\end{equation*}
One would then need to show that
$\|u\|_V \, \|\nabla(|\nabla u|^{p_1-2}\nabla u)+|u|^{p_2-2} u\|_{V^*}$ 
is locally integrable.
To ensure this in general we need, in particular, that 
\begin{equation*}
\|u\|_{W^{1}_{p_1}(\mathscr{D})}\,
\||u|^{p_2-2} u\|_{L_{p_2}(\mathscr{D})^*}=
\|u\|_{W^{1}_{p_1}(\mathscr{D})}\,
\|u\|^{p_2-1}_{L_{p_2}(\mathscr{D})}
\end{equation*}
is locally integrable, which we may not have 
if $p_1 < p_2$.  
Thus one cannot apply the It\^o formula from Gy\"ongy and Krylov.
On the other hand it is easy to check that the assumptions of 
Theorem~\ref{thm:1} are satisfied. 
\end{example}

An application of the above It\^o's formula to SPDEs driven by 
Wiener processes is given in \cite{pardoux:thesis} (Chapter 2, 
Example 5.1) and in \cite{GSS}. Further examples can be found in 
\cite[Chapter 2, Section 1.7]{Lions}.

\section{Preliminaries}

\begin{lemma}                                                                                          \label{lemma:1}
For $r\in [0,\infty)$ let $\beta(r) := \inf\{ t\geq 0: A(t) \geq r\}$ and let 
$x(t)$ be a real valued process that is locally integrable with respect to 
$dA$	for all $\omega \in \Omega$. 
Then
\begin{enumerate}[i)]
\item $\beta(r)$ is a stopping time (not necessarily finite) for every $r\in [0,\infty)$,
\item 
\begin{equation*}
\begin{split}
\int_{(0,t]} x(s)\,dA(s) & = \int_{(0,A(t)]} x(\beta(r))\, dr,		\\
\int_{(0,t)} x(s)\,dA(s) & = \int_{(0,A(t-)]} x(\beta(r)) \,dr
\end{split}
\end{equation*}
for every $t\in [0,\infty)$,
\item 
\begin{equation*}
A(\beta(t)-) - A(\beta(s)) \leq t-s	
\end{equation*}
for every $s,t \in [0,\infty)$. 
\item If $0 = r^n_0 < r^n_1 < \cdots < r^n_k < \cdots $ is an increasing sequence
of decompositions of $[0,\infty)$ such that $\sup_{k}|r^n_{k+1} - r^n_k| \to 0$
as $n\to \infty$ then for every $t\geq 0$ and $\omega \in \Omega$
\begin{equation*}
\sum_k\left|X(\tau^n_{k+1}\wedge t) - X(\tau^n_k \wedge t)\right|^2 \to \sum_{s\leq t} |X(s)|^2 |\Delta A(s)|^2	
\end{equation*}
as $n\to \infty$, 
where $X(t):= \int_{(0,t]}x(s)dA(s)$ and $\tau^n_k := \beta(r^n_k)$.
\end{enumerate}
\end{lemma}
This Lemma is proved in 
Gy\"ongy and Krylov~\cite[Lemma 1]{gyongy:krylov:on:stochastic:II}. 

Let $\kappa_n^{(j)}$ for $j=1,2$ and integers $n\geq1$ denote 
 the functions defined by 
 $$
 \kappa^{(1)}_n(t)=2^{-n}\lfloor 2^nt\rfloor, 
 \quad 
 \kappa^{(2)}_n(t)=2^{-n}\lceil 2^nt\rceil
 $$
The following lemma is known and the authors 
believe is due to Doob.  
\begin{lemma}
\label{lemma:2}
For integers $i\geq1$ let $(X_i,\|\cdot\|_{X_i})$ be Banach spaces,  
and let $p_i \in [1,\infty)$. 
Let $x_i: \R  \times \Omega \to X_i$ be $ \mathscr{B}(\R) \otimes \F$ Bochner-measurable 
such that $x_i(r)= 0$ for $r\notin [0,1]$ and
\begin{equation*}
\alpha_i:=\E\int_{0}^{1} \|x_i(r)\|_{X_i}^{p_i} \,dr < \infty.	
\end{equation*}
Then there exists a subsequence $n_k \to \infty$  such that 
for $dt$-almost all $t\in [0,1]$
\begin{equation*}
\E \int_{(0,1]}\|x_i(r) - x_i(\kappa^{(j)}_{n_k}(r-t)+t)\|_{X_i}^{p_i}\, dr 
\to 0
\,\,\textrm{ as }\,\, k\to \infty	
\end{equation*}
for $j=1,2$ and all $i\geq1$. 
\end{lemma}
\begin{proof} 
Let $(c_i)_{i=1}^{\infty}$ be a sequence of positive numbers such that 
$$
\sum_{i=1}^{\infty}c_i2^{p_i}\alpha_i<\infty. 
$$
By change of variables and changing the order of integration 
$$
I_n:=
\sum_{i=1}^{\infty}c_i\int_0^1
\E \int_{0}^{1}\|x_i(r) - x_i(\kappa^{(j)}_{n}(r-t)+t)\|_{X_i}^{p_i}\,dr\,dt 
$$
$$
\leq \sum_{i=1}^{\infty}c_i
\E\int_{-1}^1\int_{0}^1\|x_i(s+t)-x_i(\kappa^{(j)}_n(s)+t)\|^{p_i}_{X_i}\,dt\,ds. 
$$
Note that by the shift invariance of the Lebesgue measure
$$
J_{in}(s):=\int_{0}^1\|x_i(s+t)-x_i(\kappa^{(j)}_n(s)+t)\|^{p_i}_{X_i}\,dt\to 0\,(a.s.)
$$
for $s\in(0,1)$,  $i\geq1$,  
and  
$$
\sum_{i=1}^{\infty}c_i|J_{in}(s)|
\leq \sum_{i=1}^{\infty}c_i2^{p_i-1}\left(\int_{0}^1\|x_i(s+t)\|^{p_i}_{X_i}\,dt
+\int_{0}^1\|x_i(\kappa_n(s)+t)\|^{p_i}_{X_i}\,dt\right)
$$
$$
\leq \sum_{i=1}^{\infty}{c_i}2^{p_i}
\int_{0}^1\|x_i(t)\|^{p_i}_{X_i}\,dt. 
$$ 
Therefore by Lebesgue's theorem on dominated convergence 
$$
I_n=\int_0^1\left(\sum_{i=1}^{\infty}c_i
\E \int_{0}^{1}\|x_i(r) - x_i(\kappa^{(j)}_{n}(r-t)+t)\|_{X_i}^{p_i}\,dr\right)\,dt \to0. 
$$
Hence for a subsequence $n_k\to\infty$ 
$$
\sum_{i=1}^{\infty}c_i
\E \int_{0}^{1}\|x_i(r) - x_i(\kappa^{(j)}_{n}(r-t)+t)\|_{X_i}^{p_i}\,dr \to0
$$
for almost all $t\in[0,1]$, and the statement of the lemma follows. 
\end{proof}

The following lemma is proved in 
Gy\"ongy and Krylov~\cite[Lemma 3]{gyongy:krylov:on:stochastic:II}. 

\begin{lemma}
\label{lemma:3}
Let $(\xi_n)_{n\in \N}$ be a sequence of $H$-valued predictable processes. 
Suppose 
\begin{equation*}
\P \left[\sup_{n\in \N, t\leq T} |\xi_n(t)| < \infty \right]	 = 1 
\end{equation*}
and
\begin{equation*}
\P\left[ \forall t \leq T,\,\,\forall \varphi \in H\,\,\,
 \lim_{n\to \infty} (\xi_n(t),\varphi) = 0\right] = 1.	
\end{equation*}
Then for any $\varepsilon > 0$ 
\begin{equation*}
\P \left[
\sup_{t\leq T} \left|\int_{(0,t]}(\xi_n(s), dh(s))\right| > \varepsilon 
\right] \to 0	
\end{equation*}
as $n\to \infty$.
\end{lemma}

\section{Proof of the Main Result}                                           \label{sec:proof}

The following standard steps, as in 
Krylov and Rozovskii~\cite{krylov:rozovskii:stochastic}, 
allow us to work under more convenient assumptions 
without any loss of generality. 

\begin{enumerate}[1)]
\item We note that $\tau$ can be assumed to be a bounded stopping time.
Indeed if we prove Theorem~\ref{thm:1} under this assumption then we can
extend it to unbounded stopping times by considering $\tau \wedge n$ 
and letting $n\to \infty$. In fact using a non-random time change 
we may assume that $\tau\leq 1$.	
\item Recall the processes $\eta_i$ from assumption 
\eqref{assumption main}, and set 
$$
Q_i(t)=\left(\int_{(0,t]} \eta_i^{q_i}(s)\,dA(s)\right)^{1/q_i}\quad t\geq0
$$
when $q_i<\infty$, and for $q_i=\infty$ let $Q_i=(Q_i(t))_{t\geq0}$ denote a nondecreasing 
 cadlag adapted process such that almost surely
 $$
\text{$dA$-ess\,sup}_{s\leq t}\eta_i(s)\leq Q_i(t)\quad \text{ for all $t\geq0$}.
$$ 
It is not difficult 
to see that such a process $Q_i$ exists, we can take, e.g., the adapted right-continuous 
modification of the process $\text{$dA$-ess\,sup}_{s\leq t}\eta_i(s)$, i.e.,
$$
\lim_{n\to\infty}\text{$dA$-ess\,sup}_{s\leq t+1/n}\eta_i(s). 
$$
Let $(e^j)_{j\in \N} \subset V$ be an orthonormal basis in $H$ and define  
\begin{equation}
\label{eq:r_def}
\begin{split}
& r(t):=|h(0)|+A(t)
+ \sum_{i=1}^m\left(\int_{(0,t]} \|v(s)\|_{V_i}^{p_i} dA(s)\right)^{1/p_i} \\
& + \sum_{i=1}^mQ_i(t)
+ \sum_{i=1}^m\sum_{k\in \N} 2^{-c_k} \left(\int_{(0,t]}\|w_k(s)\|^{p_i}_{V_i} dA(s)\right)^{1/p_i},  
\end{split}
\end{equation}
with $c_k:=\max_{1\leq i\leq m}\sum_{j\leq k}|e_j|^2_{V_i}$ and 
$w_k := \Pi^k h$, 
where $\Pi^k$ denotes the orthogonal projection of $H$ 
onto its subspace spanned by $(e_i)_{i=1}^k$. 
We may and will assume, without loss of generality,
that $r$ and $\langle h \rangle$ are bounded.
Indeed, imagine we have proved Theorem~\ref{thm:1} under this assumption.
Consider
\begin{equation*}
\tau_n := \inf\{t \geq 0: r(t) \geq n\}. 
\end{equation*}
Then $\tau_n$ is a stopping time and $\tau_n\to\infty$ 
for $n\to\infty$. Since $\langle h\rangle$ is a predictable 
process starting from $0$, there is an increasing 
 sequence of stopping 
times $\sigma_n$ such that $\sigma_n\to\infty$ 
and $\langle h\rangle_t\leq n$
for $t\in[0,\sigma_n]$. Therefore 
$\tau_n \wedge \sigma_n \wedge \tau \to \tau$ as $n\to \infty$, 
and for fixed $n$ we get 
$r(t)\leq n$ for $t\in(0,\tau_n\wedge\sigma_n)$ 
and $\langle h \rangle_t \leq n$ for $t\in[0,\tau_n\wedge\sigma_n]$.   
Thus we get~\eqref{eq:2} and~\eqref{eq:3} 
for the stopping time $\tau_n \wedge \sigma_n \wedge \tau$ 
in place of $\tau$. 
Letting $n\to \infty$ provides~\eqref{eq:2} and~\eqref{eq:3} 
for $\tau$. 
Thus we may assume that there is $n\geq1$ such that 
$r(t)\leq n$ for $t\in(0,\tau)$ and $\langle h\rangle_t\leq n$ 
for $t\in[0,\tau]$. Moreover, by taking  
$h{\bf1}_{|h(0)|<n}$, $v{\bf1}_{|h(0)|<n}$ and 
$A{\bf1}_{|h(0)|<n}$ in place of $h$, $v$ and 
$A$, respectively, and then taking $n\to\infty$,  
 we may assume that 
$r(t)\leq n$ for $t\in [0,\tau)$ and $\langle h\rangle_t\leq n$ for $t\in[0,\tau]$. 
Furthermore, we can define $A(t) := A(\tau-)$, $h(t)=h(\tau)$, $v(t) = 0$
and $v_i^{\ast}(t) = 0$ for $t\geq \tau$. 
Then $r(t) \leq n$ and $\langle h \rangle_t \leq n$ for $t\in [0,\infty)$.
\item Finally, we can assume that 
$r(t) \leq 1$ for $t\in [0,\tau)$ and $\langle h \rangle_t\leq 1$ for 
$t\in [0,\tau]$.
Indeed let $v_n := n^{-1}v$, $A_n := n^{-1}A$ and $h_n := n^{-1}h$.
Then $r_n$, defined analogously to $r$ in~\eqref{eq:r_def} but with $v$, $A$ 
and $h$ replaced by $v_n$, $A_n$ and $h_n$ respectively,
satisfies $r_n(t) \leq n^{-1}r(t) \leq 1$.
We thus get~\eqref{eq:2} and~\eqref{eq:3} with $v$, $A$ 
and $h$ replaced by $v_n$, $A_n$ and $h_n$ respectively. 
We can now multiply by $n$ and $n^2$ to obtain the desired conclusions. 
\end{enumerate}
Now we proceed to prove Theorem~\ref{thm:1} under the assumption that
$\tau \leq 1$, $r(t) \leq 1$ and $\langle h \rangle_t \leq 1$ 
for $t\in [0,\infty)$.
Our approach is the same as in 
Gy\"ongy and Krylov~\cite{gyongy:krylov:on:stochastic:II}. 
The idea is to approximate $v$ by simple processes 
whose jumps happen at stopping times where equation~\eqref{eq:1} 
holds. 
But~\eqref{eq:1} only holds for every $\varphi \in V$ 
and $dA\times \P$ almost all
$(t,\omega) \in \rrbracket0,\tau\llbracket$, and thus it is not immediately 
clear how to choose an appropriate piecewise constant approximation to $v$. 
Here and later on for stopping times $\tau$ the notation $\rrbracket0,\tau\llbracket$ 
means the stochastic interval 
$\{(t,\omega):t\in(0,\tau(\omega)),\omega\in\Omega\}$.

\begin{proposition}                                                                                    \label{propn:1}
There is a nested sequence of  random partitions of $[0,\infty]$, 
\begin{equation*}
0 = \tau^n_0 < \tau^n_1 \leq \tau^n_2 \leq \cdots \leq \tau^n_{N(n)+1} = \infty,	
\end{equation*}
with stopping times  $\tau^n_j$, $j=1,\ldots,N(n)+1$, such that 
for every $\omega \in \Omega$ either $\tau^n_j(\omega) < \tau(\omega)$
or $\tau^n_j(\omega) = \infty$, and such that the following statements hold. 
\begin{enumerate}[(1)]
\item There is $\Omega' \subset \Omega$ such that $\P(\Omega') = 1$ 
and with 
\begin{equation*}
I(\omega):=\{\tau^n_j(\omega):n\in \N, j=1,\ldots,N(n)\}\cap (0,\infty)
\end{equation*}
we have~\eqref{eq:1} satisfied for every 
$\omega \in \Omega'$, $t\in I(\omega)$ 
and $\varphi \in V$. Moreover, if $\Delta A(t)>0$ for 
some $t>0$ and $\omega\in\Omega'$, then $t\in I(\omega)$. 
Furthermore, if $0\leq s<t$ and $(s,t]\cap I(\omega)=\emptyset$, then 
$A(s)=A(t)$. 
\item 
For $l\in\{1,2\}$, $i=1,\ldots,m$ and for all $k\geq1$ 
\begin{equation}                                                                         \label{eq:lims1:in_proposition}
\begin{split}
\lim_{n\to \infty} \E\int_{(0,\infty)}
\|v(s) - v^{(l)}_n(s)\|_{V_i}^{p_i}\,dA(s) = 0,\\
\lim_{n\to \infty} \E\int_{(0,\infty)}
\|w_k(s) - w^{(l)}_{kn}(s)\|^{p_i}_{V_i}\,dA(s) = 0, 
\end{split}	
\end{equation}
where 
$$
v_n^{(1)}(t):=\sum_{j=1}^{N(n)}v(\tau_j^n){\bf1}_{[\tau_j^n,\tau_{j+1}^n)}(t), 
\quad 
v_n^{(2)}(t):=\sum_{j=0}^{N(n)}v(\tau_{j+1}^n){\bf1}_{(\tau_j^n,\tau_{j+1}^n]}(t), 
$$
and $w^{(l)}_{kn}$ is defined analogously from $w_k=\Pi^kh$. 
\end{enumerate}
\end{proposition}
\begin{proof}
Since $V$ is separable there is $\{\varphi_i\}_{i\in \N}\subset V$ 
which is dense in $V$.
For each $\varphi_i$ there is an exceptional set 
$D_i \in [0,\infty)\times\Omega$ such that~\eqref{eq:1}
holds for $(t,\omega)\in\rrbracket0,\tau\llbracket\setminus D_i$
and $(dA\times \P)(D_i)=0$.
Let $D = \bigcup_{i\in \N} D_i$. 
Then $(dA \times \P)(D)=0$ and~\eqref{eq:1} holds for all $\varphi\in V$ 
and all $(t,\omega)\in\rrbracket0,\tau\llbracket\setminus D$.
Now using Lemma~\ref{lemma:1} and the Fubini theorem
\begin{equation*}
\begin{split}
0 & = \E\int_{(0,\tau)} \chi_D(s)\,dA(s) 
= \E \int_{(0,A(\tau-)]} \chi_D(\beta(r))\,dr\\
& = \int_{(0,\infty)} \P(r\leq A(\tau), (\beta(r),\omega)\in D)\,dr. 		
\end{split}
\end{equation*}
From this we see that for $dr$ almost all $r\in (0,\infty)$ there 
is $\Omega(r) \subset \Omega$ with $\P(\Omega(r)) = 1$ such that for 
any $\omega \in \Omega(r)$ either $r > A(\tau(\omega),\omega)$ 
or $\beta(r,\omega) < \tau(\omega)$ and for $t=\beta(r)$ and for
all $\varphi \in V$
\begin{equation}
\label{eq:1_again}
( v(t), \varphi ) = \sum_{i=1}^m \int_{(0,t]} \langle v_i^{\ast}(s),\varphi \rangle \,dA(s) 
+ ( h(t),\varphi ).	
\end{equation}
By virtue of Lemma~\ref{lemma:2} there is a nested sequence of 
decompositions of $[0,1]$, 
\begin{equation}
\label{eq r}
0 = r^n_0 < r^n_1 < \cdots < r^n_{N(n)+1} = 1, 	
\end{equation}
such that $\lim_{n\to \infty} \max_i |r^n_{j+1} - r^n_j| = 0$,  
and 
\begin{equation}
\label{eq:lims_with_beta}
\begin{split}
\lim_{n\to \infty} \E\int_0^1
\|v(\beta(r)) - v(\beta(\kappa_n^{(l)}(r)))\|_{V_i}^{p_i}\,dr = 0,\\
\lim_{n\to \infty} \E\int_0^1
\|w_k(\beta(r)) - w_k(\beta(\kappa_n^{(l)}(r)))\|^{p_i}_{V_i}\,dr = 0   
\end{split}	
\end{equation}
for all $i=1,\ldots,m$, all $k\in \N$ and $l=1,2$, where 
$\kappa_n^{(1)}(r) = r^n_j$ if $r\in [r^n_j,r^n_{j+1})$
and $\kappa_n^{(2)}(r) = r^n_{j+1}$  if $r\in (r^n_j,r^n_{j+1}]$.

Now let  $\Omega' := \bigcap_{n\in \N} 
\left( \Omega(r^n_0)\cap \ldots \cap\Omega(r^n_{N(n)+1})\right)$,  
$\tau^n_j := \beta(r^n_j)$, and 
\begin{equation*}
I(\omega):=\{\tau^n_i(\omega):n\in \N, i=1,\ldots,N(n)\}\cap (0,\infty).	
\end{equation*}
Then $\P(\Omega') = 1$ and 
\begin{equation*}
0 = \tau^n_0 < \tau^n_1 \leq \tau^n_2 \leq \cdots \leq \tau^n_{N(n)+1} = \infty,	\quad n=1,2,\ldots,
\end{equation*}
is a nested sequence of random partitions of $(0,1)$ by stopping times 
$\tau^n_j$ such that  
statement (1) holds. 
To prove (2) we notice that, 
just like in \cite{gyongy:krylov:on:stochastic:II}, 
for $r\in(r_j^n,r^n_{j+1}]$ 
$$
v_n^{(2)}(\beta(r))=
\left\{ 
\begin{array}{lll}
v(\beta(r^n_{j+1})) = v(\beta(\kappa^{(2)}_n(r))) & \text{if} & \beta(r^n_j)  < \beta(r) \\
v(\beta(r^n_j)) = v(\beta(\kappa^{(1)}_n(r))) & \text{if} & \beta(r^n_j)  = \beta(r).
\end{array}
\right.
$$
Thus with appropriate sets $S_n\in\mathcal B(\R)\times\mathcal F$ 
$$
v^{(2)}(\beta(r))={\bf1}_{S_n}(r)v(\beta(\kappa^{(2)}_n(r))) - (1-{\bf1}_{S_n})(r))v(\beta(\kappa^{(1)}_n(r))).
$$
Hence due to~\eqref{eq:lims_with_beta} and Lemma~\ref{lemma:1} we obtain
the first equality in~\eqref{eq:lims1:in_proposition} for $l=2$, 
$i=1,\ldots,m$ and for all $k\in \N$.
The rest of~\eqref{eq:lims1:in_proposition} is obtained similarly. 
\end{proof}
\begin{proposition}                                   \label{proposition 2}
For every $n\in \N$,
every $\omega \in \Omega'$
and every $\tau^n_j(\omega) \in I(\omega)$
\begin{equation}                                                      \label{eq:towards_ito_3:in_proposition} 
\begin{split}
|v(\tau^n_j)|^2 = & |h(0)|^2	
+ 2\sum_{i=1}^m \int_{(0,\tau^n_j]} \langle v_i^{\ast}(s), v^{(2)}_n(s) \rangle\,dA(s) \\
& + 2 \int_{(0,\tau^n_j]} \bar{v}_n(s)\, dh(s)	 + 2(h(0), h(\tau^n_1) - h(0)) \\
& + \sum_{k=0}^{j-1}|h(\tau^n_{k+1})
-h(\tau^n_k)|^2 - |v(\tau^n_1)-h(\tau^n_1)|^2\\
& - \sum_{k=1}^{j-1} |v(\tau^n_{k+1})-v(\tau^n_k)
 - (h(\tau^n_{k+1}) - h(\tau^n_k))|^2, 
\end{split}
\end{equation}
where $\bar v_n(s)=0$ for $s\in[0,\tau^n_1]$ and 
$\bar v_n(s)=v(\tau^n_j)$ for $s\in(\tau^n_j,\tau^n_{j+1}]$ for $j=1,\ldots,N(n)$. 
Moreover, 
\begin{equation}                                                            \label{eq:bnds_3}
\E \sup_{t\in I} |v(t)|^2 < \infty.	
\end{equation}
\end{proposition}
\begin{proof}
Let $\omega \in \Omega'$ and $t,t' \in I(\omega)$ and $t' \geq t$. 
Clearly, 
$$
|v(t')|^2-|v(t)|^2=2(v(t'),v(t')-v(t))-|v(t')-v(t)|^2, 
$$
which by statement (1) of Proposition \ref{propn:1} gives 
\begin{equation*}
\begin{split}
& |v(t')|^2 - |v(t)|^2 \\
& = 2\sum_{i=1}^m \int_{(t,t']} \langle v_i^{\ast}(s), v(t') \rangle \,dA(s) 
+ 2(h(t')-h(t),v(t')) - |v(t')-v(t)|^2 .	
\end{split}
\end{equation*}  
Hence by the identity 
\begin{equation*}
\begin{split}
& 2( h(t')-h(t), v(t')-v(t)) \\
& = - |v(t')-v(t) - (h(t')-h(t))|^2 + |v(t')-v(t)|^2 + |h(t')-h(t)|^2,  	
\end{split}
\end{equation*}
we have 
\begin{equation}                                                                            \label{eq:towards_ito_2}
\begin{split}
|v(t')|^2 & - |v(t)|^2 
= 2\sum_{i=1}^m \int_{(t,t']} \langle v_i^{\ast}(s), v(t') \rangle dA(s) 
+ 2(v(t), h(t')-h(t))\\
& + |h(t')-h(t)|^2 - |v(t')-v(t) - (h(t')-h(t))|^2.  	
\end{split}
\end{equation}
By (1) in Proposition \ref{propn:1} again 
\begin{equation*}
2|v(t)|^2 = 2\sum_{i=1}^m \int_{(0,t]} \langle v_i^{\ast}(s), v(t) \rangle\, dA(s) 
+ 2(h(t),v(t)),
\end{equation*}
which by the identity 
$2(h(t),v(t)) = -|v(t)-h(t)|^2 + |v(t)|^2 + |h(t)|^2$ 
gives 
\begin{equation}                                                                    \label{eq:towards_ito_1}
|v(t)|^2 = 2\sum_{i=1}^m \int_{(0,t]} \langle v_i^{\ast}(s), v(t) \rangle\, \,dA(s) 
+ |h(t)|^2 - |v(t)-h(t)|^2.	
\end{equation}
Summing up for $k=1,\ldots,j-1$ equations \eqref{eq:towards_ito_2} 
with $t'=\tau^n_{k+1}$, $t=\tau^n_{k}$,  
and adding to it equation \eqref{eq:towards_ito_1} with $t=\tau^n_1$, we 
obtain \eqref{eq:towards_ito_3:in_proposition}. 
Form \eqref{eq:towards_ito_3:in_proposition} we have 
\begin{equation*}
\begin{split}
\E \max_{1\leq j \leq N(n)}	|v(\tau^n_j)|^2 
\leq & 2\E|h(0)|^2 
+ 2\E \sum_{i=1}^m \int_{(0,\tau]} |\langle v_i^{\ast}(s), v^{(2)}_n(s) \rangle|\,dA(s)\\
& + 2\E \max_{1\leq j \leq N(n)} 
\left|\int_{(0,\tau^n_j]}\bar{v}_n(s)\,dh(s)\right|\\ 
& + 2\E 
\sum_{k=0}^{N(n)}|h(\tau^n_{k+1}) -h(\tau^n_k)|^2	.
\end{split}
\end{equation*}
Clearly
$$
2\E \max_{1\leq j \leq N(n)} \left|\int_{(0,\tau^n_j]}\bar{v}_n(s)\,dh(s)\right|	
\leq 16 + \frac{1}{16} \E \sup_{t\geq 0} \left|\int_{(0,t]}\bar{v}_n(s)\,dh(s)\right|^2, 	
$$
and by 
Doob's inequality and $\langle h\rangle\leq 1$, 
\begin{equation*}
\E \sup_{t\geq 0} \left|\int_{(0,t]}\bar{v}_n(s)\,dh(s)\right|^2\leq 
4 \E \int_{0}^{\infty} |\bar{v}_n(s)|^2 d\langle h \rangle_s 
\leq4\E\max_{1\leq j \leq N(n)} 	|v(\tau^n_j)|^2. 
\end{equation*}
Since $h$ is a martingale, 
$$
\E 
\sum_{k=0}^{N(n)}|h(\tau^n_{k+1}) -h(\tau^n_k)|^2\leq \E|h(1)|^2
=\E\langle h\rangle(1)\leq1. 
$$
By H\"older's inequality and $\sum_iQ_i\leq1$ we have 
$$
\sum_{i=1}^m\E\int_{(0,\tau]} |\langle v_i^{\ast}(s), v^{(2)}_n(s) \rangle|\,dA(s)
$$
$$
\leq\sum_{i}\sup_{n\geq 1}
\left(\E\int_{(0,\tau]} \|v^{(2)}_n(s)\|_{V_i}^{p_i}\, dA(s)\right)^{\tfrac{1}{p_i}}=:c\,, 
$$
which by virtue of \eqref{eq:lims1:in_proposition} is finite. Hence, taking also 
into account $\E|h(0)|^2\leq1$ we have 
\begin{equation*}
\E \max_{1\leq j \leq N(n)}	|v(\tau^n_j)|^2 
\leq  22+2c+
\frac{1}{4} \E \max_{1\leq j \leq N(n)}	|v(\tau^n_j)|^2,  
\end{equation*}
which immediately yields \eqref{eq:bnds_3}, provided 
\begin{equation}                                               \label{finite}
\E \max_{1\leq j \leq N(n)}	|v(\tau^n_j)|^2<\infty. 
\end{equation}
To show \eqref{finite} 
note that due to~\eqref{eq:towards_ito_1}, 
for every $n\in \N$
and $j=1,\ldots,N(n)+1$,
we get
\begin{equation}                                                  \label{HV}
\begin{split}
\E|v(\tau^n_j)|^2 
\leq & 
\E |h(\tau^n_j)|^2 
+ 2\E \sum_{i=1}^m \int_{(0,\tau^n_j]}\langle v_i^{\ast}(s), v(\tau^n_j)\rangle dA(s)\\
\leq & 
\E |h(0)|^2 
+ 2 \E \sum_{i=1}^m 
Q_i(\tau)
\left(\int_{(0,\tau]} \|v(\tau^n_j)\|_{V_i}^{p_i} dA(s)\right)^{\tfrac{1}{p_i}}\\
\leq & 
1 + 2\sum_i\E \|v(\tau^n_j)\|_{V_i},
\end{split}
\end{equation}
since $\tau \leq 1$ and $r(t) \leq1$ for all $t\in [0,\infty)$.
For $i=1,\ldots,m$ 
\begin{equation*}
\begin{split}
& 
\E \|v(\tau^n_j)\|^{p_i}_{V_i}\leq \E \sup_{s\in [0,\infty)} \|v^{(2)}_n(s)\|_{V_i}^{p_i}
\leq \E \sup_{r\in (0,1]} \|v^{(2)}_n(\beta(r))\|_{V_i}^{p_i}\\
&\leq 2^{p_i-1}\sum_{l=1}^2\E \sup_{r\in (0,1]} \|v(\beta(\kappa^{(l)}_n(r)))\|_{V_i}^{p_i}\\
&\leq 2^{p_i-1} \sum_{l=1}^2
\E \sum_{k=0}^{N(n)}\frac{1}{r^n_{k+1}-r^n_k}
\int_{r^n_k}^{r^n_{k+1}} \|v(\beta(\kappa^{(l)}_n(r)))\|_{V_i}^{p_i}\, dr	\\
& \leq \frac{2^{p_i-1}}{d_n} 
\sum_{l=1}^2 
\E \int_0^1 \|v(\beta(\kappa^{(l)}_n(r)))\|_{V_i}^{p_i}\, dr	
< 2^{p_i}\frac{c_i}{d_n},\\
\end{split}	
\end{equation*}
where $r^n_k$ are given by~\eqref{eq r}, 
$d_n := \min_{k=1,\ldots,N(n)}|r^n_{k+1}-r^n_k|>0$ and 
\begin{equation*}                                                                                    
c_i:=\max_l\sup_n\sum_{i=1}^m\E \int_0^1 
\|v(\beta(\kappa^{(l)}_n(r)))\|_{V_i}^{p_i}\, dr,  	
\end{equation*}
which due to~\eqref{eq:lims_with_beta} is finite. 
Hence by virtue of \eqref{HV} we have \eqref{finite}, 
which completes the proof of \eqref{eq:bnds_3}. 
\end{proof}

We see that due to~\eqref{eq:bnds_3} 
there is $\Omega'' \subset \Omega'$ 
such that $\P(\Omega'') = 1$ and 
\begin{equation*}
\sup_{t\in I(\omega)} |v(t)|^2 < \infty\quad\text{for all $\omega \in \Omega''$}. 
\end{equation*}
Moreover, since $h$ is cadlag, for all $\omega \in \Omega''$ we have 
\begin{equation}                                                  \label{eq:bnds_4}
\sup_{t\in I(\omega)} |v(t)-h(t)|^2 < \infty.	
\end{equation}
Define 
\begin{equation}
\label{eq:def_of_z}
z^{(1)}(t) := \int_{(0,t)} \sum_{i=1}^mv_i^{\ast}(s)\,dA(s),\quad 
z^{(2)}(t) := \int_{(0,t]} \sum_{i=1}^mv_i^{\ast}(s)\,dA(s), 	
\end{equation}
for $t\geq0$, where the integrals are defined as weak* integrals. 
Recall that $v^{\ast}=\sum_iv^{\ast}$ is a $V^{\ast}$-valued 
such that $\langle v^{\ast}(t),\varphi\rangle $ is a progressively measurable process 
for every $\varphi\in V$, and 
$$
\int_{(0,t]}|\langle v^{\ast}(s),\varphi\rangle|\,dA(s)\leq 
\sum_i\int_{(0,t]}|\langle v_i^{\ast}(s),\varphi\rangle|\,dA(s)
$$
$$
\leq \sum_i|\varphi|_{V_i}\int_{(0,t]}\eta_i(s)\,dA(s)
\leq |\varphi|_{V}\sum_i\int_{(0,t]}\eta_i(s)\,dA(s)<\infty. 
$$
Therefore $z^{(1)}$ and $z^{(2)}$ are well-defined $V^{\ast}$-valued 
processes such that  $\langle z^{(1)}, \varphi\rangle$ and 
$\langle z^{(2)}, \varphi\rangle$ are left-continuous 
and right-continuous adapted processes, respectively. 

In what follows we use the notation 
$\Delta^w f(t):=f(t)-\text{w-lim}_{s\nearrow t}f(s)$ for $H$-valued functions 
$f$, when the weak limit from the left exists at $t$. 

\begin{proposition}                                                 \label{propn:2}
Let $z^{(l)}$, $l\in \{1,2\}$ be given by~\eqref{eq:def_of_z}.
\begin{enumerate}
\item If $\omega \in \Omega''$ and $t\in (0,\infty)$ then $z^{(l)}(t)\in H$ 
for  $l\in\{1,2\}$.
Moreover 
\begin{equation*}
\sup_{t\in (0,\infty)} |z^{(l)}(t)| < \infty  \,\,\,\, 
\forall \omega \in \Omega'',\,\,	 l \in \{1,2\}.
\end{equation*}
\item Let $\tilde v$ be given by
\begin{equation*}
\tilde v(t) := \chi_{\Omega''} z^{(2)}(t) + h(t).	
\end{equation*}
Then $\tilde v$ is a $H$-valued adapted and weakly cadlag process 
such that $v(t) = \tilde v(t)$ for all $t\in I(\omega)$ and $\omega \in \Omega''$.
Moreover
\begin{equation*}
\sup_{t\in (0,\infty)} |\tilde v(t)| < \infty  \,\,\,\, 
\forall \omega \in \Omega''.
\end{equation*}
\item If $\omega \in \Omega''$ then for all $t\in (0,\tau(\omega))$
\begin{equation}
\label{eq:Delta_w}
\Delta^w (\tilde v-h)(t) = (\Delta A)(t)\sum_{i=1}^m v_i^{\ast}(t).	
\end{equation}
\end{enumerate}
\end{proposition}
\begin{proof}
Fix $\omega \in \Omega''$. 
If $t\in I(\omega)$ then for all $\varphi \in V$ 
\begin{equation*}
\left(v(t) - h(t), \varphi\right) 
= \sum_{i=1}^m\int_{(0,t]} \langle v_i^{\ast}(s),\varphi \rangle\,dA(s), 	
\end{equation*}
and hence $z^{(2)}(t) \in H$.
Consider now the situation when $t\in (0,\tau(\omega)] \setminus I(\omega)$. 
Let $\bar{I}^l(\omega)$ denote the left-closure of the set $I(\omega)$.
If $t\in \bar{I}^l(\omega)\setminus I(\omega)$ then $\Delta A(t)=0$ 
by Proposition \ref{propn:1}, and there is a sequence 
$(t_n)_{n\in \N} \subset I(\omega)$ such that $t_n \nearrow t$. Moreover, 
due to~\eqref{eq:bnds_4}
there is a subsequence $t_{n'}\nearrow t$  such that 
$v(t_{n'}) - h(t_{n'})$ converges weakly in $H$ to some $\xi\in H$.
Hence 
for all $\varphi \in V$
\begin{equation*}
\begin{split}
(\xi, \varphi) & = \lim_{n'\to \infty} (v(t_{n'}) - h(t_{n'}),\varphi)\\
& = \lim_{n'\to \infty} \sum_{i=1}^m \int_{(0,t_{n'}]} 
\langle v_i^{\ast}(s), \varphi	\rangle\,dA(s) 
= \sum_{i=1}^m \int_{(0,t)} \langle v_i^{\ast}(s), \varphi	\rangle\,dA(s)	 \\
& = \sum_{i=1}^m \int_{(0,t]} \langle v_i^{\ast}(s), \varphi	\rangle\,dA(s), 
\end{split}
\end{equation*}
which implies $z^{(2)}(t)=\xi \in H$.
If $t\in (0,\infty)\setminus\bar{I}^l(\omega)$, then 
there is $s \in \{0\}\cup\bar{I}^l(\omega)$ such that
$s<t$ and $(s,t] \cap I(\omega) = \emptyset$. 
So $\int_{(s,t]} v_i^{\ast}(s)\,dA(s) = 0$ and $z^{(2)}(t) = z^{(2)}(s) \in H$.
Of course if $t=0$ then $z^{(2)}(t) = 0 \in H$.
Finally, due to~\eqref{eq:bnds_4}, 
\begin{equation}
\label{eq:bnds_5}
\sup_{t\in (0,\infty)}|z^{(2)}(t)|^2 
= \sup_{t\in (0,\infty)}|v(t)-h(t)|^2 < \infty.	
\end{equation}
Now we consider $z^{(1)}(t)$ for $t\in (0,\infty)$. 
Take $(t_n)_{n\in \N}$ such that $t_n < t$ and $t_n \nearrow t$ as $n\to \infty$.
From~\eqref{eq:bnds_5} we know that $\sup_{n\in \N} |z^{(2)}(t_n)|^2 < \infty$
and so there is a subsequence $t_{n'}\nearrow t$ such that 
$z^{(2)}(t_n)$ converges weakly in $H$ to some $\xi\in H$. 
Thus for any $\varphi \in V$ 
\begin{equation*}
\begin{split}
& (\xi, \varphi) = \lim_{n'\to \infty} (z^{(2)}(t_{n'}),\varphi)\\
&  =\lim_{n'\to \infty} \sum_{i=1}^m \int_{(0,t_{n'}]} \langle v_i^{\ast}(s),\varphi \rangle\,dA(s)
= \sum_{i=1}^m \int_{(0,t_{n'})} \langle v_i^{\ast}(s),\varphi \rangle\,dA(s) 
= \langle z^{(1)}(t), \varphi \rangle.		
\end{split}
\end{equation*}
Hence $z^{(1)}(t) = \xi \in H$, and due to~\eqref{eq:bnds_5} 
\begin{equation*}
\sup_{t\in (0,\infty)}|z^{(1)}(t)|^2 \leq \sup_{t\in (0,\infty)}|z^{(2)}(t)|^2 < \infty.
\end{equation*}
By construction $\tilde v$ is weakly cadlag.
Due to~\eqref{eq:bnds_5} for $\omega \in \Omega''$
\begin{equation*}
\sup_{t\in (0,\infty)} |\tilde v(t)|^2 
\leq \sup_{t\in (0,\infty)} |	z^{(2)}(t)|^2 + \sup_{t\in (0,\infty)} |	h(t)|^2 
< \infty.
\end{equation*}
We note that for any $\varphi \in V$ the real valued random variable
\begin{equation*}
(\tilde v(t),\varphi) 
= \chi_{\Omega''} \sum_{i=1}^m 	
\int_{(0,t]} \langle v_i^{\ast}(s), \varphi \rangle\,dA(s)
+ (h(t),\varphi)
\end{equation*}
is $\F_t$-measurable.  Hence, since $H$ is separable, 
$\tilde v(t)$ 
is $\F_t$-measurable by the Pettis theorem. 
Finally notice that 
$$
\Delta ((\tilde v-h)(t),\varphi) = \sum_{i=1}^m \langle v_i^{\ast}(t),\varphi\rangle (\Delta A)(t)
$$
for all $\varphi\in V$ and $\omega\in\Omega''$. Hence on $\Omega''$ 
$$
\Delta^w (\tilde v-h)(t)= \sum_{i=1}^m v_i^{\ast}(t)(\Delta A)(t)\,.
$$  
\end{proof}

Let
\begin{equation*}
\tilde v_n(t) := \tilde v(\tau^n_j) \,\,\, \textrm{and}\,\,\,\, 
h_n(t) := h(\tau^n_j)\,\,\,\, \textrm{for}\,\,\,\, t\in (\tau^n_j,\tau^n_{j+1}],
\,\,\,\, j = 0,1,\ldots,N(n).	
\end{equation*}
Then from~\eqref{eq:towards_ito_3:in_proposition} it follows that 
for every $\omega \in \Omega''$ and $t:=\tau^n_j(\omega) \in I(\omega)$
\begin{equation}                                      \label{eq:towards_ito_4} 
\begin{split}
|\tilde v(t)|^2 = & |h(0)|^2	
+ 2\sum_{i=1}^m \int_{(0,t]} \langle v_i^{\ast}(s), v^{(2)}_n(s) \rangle\,dA(s) \\
& + 2 \int_{(0,t]} (\tilde v_n(s), dh(s)) 
+ \sum_{k=0}^{j-1}|h(\tau^n_{k+1})-h(\tau^n_k)|^2-K_n(t),
\end{split}
\end{equation}
where
$$ 
K_n(t):= \sum_{k:\tau^n_{k+1}\leq t}^{j-1} |\tilde v(\tau^n_{k+1})-\tilde v(\tau^n_k)
 - (h(\tau^n_{k+1}) - h(\tau^n_k))|^2.
$$
In order to let $n\to\infty$ in the above equation we first rewrite it 
as 
\begin{equation}                                                                 \label{eq:towards_ito_5} 
\begin{split}
|\tilde v(t)|^2 = &  2\sum_{i=1}^m \int_{(0,t]} \langle v_i^{\ast}(s), v^{(2)}_n(s) \rangle\, dA(s) \\
& + 2 \int_{(0,t]} (\tilde v_n(s)-h_n(s), dh(s)) 
+ |h(t)|^2-K_n(t) \\
\end{split}
\end{equation}
by noticing that 
\begin{equation*}
2\int_{(0,\tau^n_j]} (h_n(s), dh(s)) 
=|h(\tau^n_j)|^2 - |h(0)|^2 - \sum_{k=0}^{j-1} |h(\tau^n_{k+1}) - h(\tau^n_j)|^2.  	
\end{equation*}

To perform the limit procedure we use the following two 
propositions. 
\begin{proposition}
\label{prop:4}
There is $\tilde\Omega \subset \Omega''$ with 
$\P(\tilde\Omega) = 1$ such that for a subsequence $n'$ 
and for every $\omega \in \tilde\Omega$ 
\begin{equation*}
\begin{split}
& \int_{(0,\infty)}
\|v(s) - v^{(l)}_{n'}(s)\|_{V_i}^{p_i}\,dA(s) \to 0\,\,\,\, (l=1,2),\\
&  \int_{(0,\infty)}
\|w_k(s) - w^{(l)}_{kn'}(s)\|^{p_i}_{V_i}\,dA(s) \to 0\,\,\,\, (l=1,2; k\in \N),\\
& \sup_{t\in (0,\infty)}\left|\int_{(0,t]} (\tilde v_{n'}(s) - h_{n'}(s),dh(s))  
- \int_{(0,t]} (\tilde v(s-)-h(s-),dh(s))\right| \to 0
\end{split}		
\end{equation*}
as $n'\to \infty$. Moreover, 
$$
K_{n'}(t)\to \int_{(0,t]}|v^{\ast}(s)|^2 \Delta A(s)\,dA(s) \quad\text{for $t\in I(\omega)$ and 
$\omega\in\tilde\Omega$}. 	
$$
\end{proposition}

\begin{proof}
Set $\xi(t) := \tilde v(t-) - h(t-)$ and $\xi_n(t):=\tilde v_n(t)-h_n(t)$. 
By Lemma~\ref{lemma:3}, taking into account that 
by Proposition~\ref{propn:2} on $\Omega''$ 
$$
\sup_n\sup_{t\in (0,\infty)} |\xi(t) - \xi_n(t)| \leq \sup_{t\in (0,\infty)}|z^{(1)}(t)| < \infty,  
$$
and that  $V$ is dense in $H$, we have 
$$
\sup_{t\geq0}\left|\int_{(0,t]}(\xi(s)-\xi_n(s),dh(s))\right|\to 0\quad
\text{in probability as $n\to\infty$}, 
$$
if we show that  
almost surely 
$$
\lim_{n\to\infty}(\xi(t)-\xi_n(t),\varphi)=0 \quad
\text{for all $t>0$ and $\varphi\in V$}.
$$ 
 To this end set 
$$
v_i^{\ast}:=\int_{(\tau^n_j,t)} v_i^{\ast}(s)\,dA(s)\in V_i^{\ast}. 
$$
Then for all $\omega\in\Omega''$, $t>0$ and $\varphi\in V$ 
\begin{equation*}
\begin{split}
& (\xi(t) - \xi_n(t),\varphi)=\langle \xi(t) - \xi_n(t),\varphi\rangle
 = \bigg\langle \sum_{i=1}^m v_i^{\ast},\varphi\bigg\rangle 
= \sum_{i=1}^m \langle v_i^{\ast},\varphi\rangle\\
 &=  \sum_{i=1}^m \int_{(\tau^n_j,t)} \langle v_i^{\ast}(s),\varphi\rangle\,dA(s) 
\leq \sum_{i=1}^{m}\|\varphi\|_{V_i} \int_{(\tau^n_j,t)} \|v_i^{\ast}(s)\|_{V_i^*}\,dA(s) 	\\
& \leq \max_{j=1,\ldots,N(n)} \sum_{i=1}^{m}\|\varphi\|_{V_i}
\left(A(\tau^n_{j+1})-A(\tau^n_j)\right)^{\frac{1}{p_i}} 
Q_i(\tau^n_{j+1})\\
& \leq \max_{j=1,\ldots,N(n)}\sum_{i=1}^{m}\|\varphi\|_{V_i}
\left|r^n_{j+1} - r^n_j\right|^{\frac{1}{p_i}} \to 0
\,\,\, \textrm{as}\,\,\, n \to \infty,
\end{split}
\end{equation*}
with $r^n_j$ given by~\eqref{eq r}.
Consequently, taking also into account \eqref{eq:lims1:in_proposition} 
of Proposition \ref{propn:1} we have $\Omega'''\subset\Omega''$ and 
a subsequence $n'$ 
such that the first three limits are zero for $\omega\in\Omega'''$. 
Taking the limit along the subsequence $n'$ in \eqref{eq:towards_ito_5}
we see that $K_{n'}(t)$ converges for $\omega\in\Omega'''$ and $t\in I(\omega)$ 
to some $K(t)$, and 
\begin{equation*}                                                                 
\begin{split}
|\tilde v(t)|^2 = &  2\sum_{i=1}^m \int_{(0,t]} \langle v_i^{\ast}(s), v(s) \rangle\, dA(s) \\
& + 2 \int_{(0,t]} (\tilde v(s)-h(s), dh(s)) 
+ |h(t)|^2-K(t). \\
\end{split}
\end{equation*}
From this point onwards we will always consider only the subsequence 
$n'$ but we will keep writing $n$ to ease notation.
Our task is now to identify $K(t)$. 
We note that, using Parseval's identity, 
\begin{equation*}
\begin{split}
K_n(t) & =
\sum_{0\leq \tau^n_{j+1} \leq t} 
\left|\sum_i\int_{(\tau^n_j,\tau^n_{j+1}]} v_i^{\ast}(s)\,dA(s) \right|^2	\\
& = \sum_{0\leq \tau^n_{j+1} \leq t}  \sum_{k\in \N} 
\left(\sum_i\int_{(\tau^n_j,\tau^n_{j+1}]} v_i^{\ast}(s)\,dA(s), e_k \right)^2\\
& = \sum_{0\leq \tau^n_{j+1} \leq t}  \sum_{k\in \N} 
\left\langle\sum_i\int_{(\tau^n_j,\tau^n_{j+1}]} v_i^{\ast}(s)\,dA(s), e_k \right\rangle^2\\
& = \sum_{0\leq \tau^n_{j+1} \leq t}  \sum_{k\in \N} 
\left|\int_{(\tau^n_j,\tau^n_{j+1}]}\sum_i\langle v_i^{\ast}(s), e_k\rangle\,dA(s) \right|^2
.
\end{split}	
\end{equation*}
Hence, using Lemma \ref{lemma:1}, Parseval's identity and~\eqref{eq:Delta_w}, we get
\begin{align}
K(t) & = \lim_{n\to \infty} K_n(t)  
\geq 
\sum_{k\in \N}   \varliminf_{n\to \infty} \sum_{0\leq \tau^n_{j+1} \leq t}
\left|\int_{(\tau^n_j,\tau^n_{j+1}]} \sum_i\langle v_i^{\ast}(s), e_k \rangle\,dA(s) \right|	^2\nonumber\\
& = \sum_{k\in \N} \sum_{s\leq t} 
\left| \sum_i\langle v_i^{\ast}(s),  e_k \rangle \Delta A(s) \right|^2
= \sum_{k\in \N} \sum_{s\leq t} 
\left| ( \Delta^w (\tilde v-h)(s),  e_k ) \right|^2                                                    \nonumber\\ 
& = \sum_{s\leq t} \left|\sum_iv_i^{\ast}(s)\right|^2 |\Delta A(s)|^2.                       \label{Knbelow}
\end{align}	
To obtain an upper bound we use first the identity  
$$
|x+y|^2= y^2 +2x(y+x)-x^2
$$
 together with the definition of $g$
to get 
\begin{equation}                                                              \label{Kn}
\begin{split}
K_n(t) & = \sum_{0\leq \tau^n_{j+1} \leq t} 
\left|\sum_i\int_{(\tau^n_j, \tau^n_{j+1})} v_i^{\ast}(s)\,dA(s) 
+ \sum_iv_i^{\ast}(\tau^n_{j+1}) \Delta A(\tau^n_{j+1})\right|^2\\
& = \sum_{0\leq \tau^n_{j+1} \leq t} (J^{(1)}_j+J^{(2)}_j-J^{(3)}_j)
\end{split}
\end{equation}
with 
\begin{equation*}
\begin{split}
J^{(1)}_j:=&\left|\sum_iv_i^{\ast}(\tau^n_{j+1})\right|^2 |\Delta A(\tau^n_{j+1})|^2, 
\\
J^{(2)}_j:=&2 \left(\sum_i
\int_{(\tau^n_j, \tau^n_{j+1})} v_i^{\ast}(s)\,dA(s), \tilde v(\tau^n_{j+1})-\tilde v(\tau^n_j) 
- (h(\tau^n_{j+1})-h(\tau^n_j)) \right), \\
J^{(3)}_j:=&\left|\sum_i\int_{(\tau^n_j, \tau^n_{j+1})} v_i^{\ast}(s)\,dA(s)\right|^2
\end{split}	
\end{equation*}
For $j\neq0$ we split  $J^{(2)}_j=J^{(21)}_j-J^{(22)}_j$ with 
$$
J^{(21)}_j:=2\left(\sum_i\int_{(\tau^n_j, \tau^n_{j+1})} v_i^{\ast}(s)\,dA(s), \tilde v(\tau^n_{j+1})-\tilde v(\tau^n_j)\right), 
$$
$$
J^{(22)}_j:=2\left(\sum_i
\int_{(\tau^n_j, \tau^n_{j+1})} v_i^{\ast}(s)\,dA(s), h(\tau^n_{j+1})-h(\tau^n_j) \right), 
$$
and notice that 
$$
J^{(21)}_j=2\sum_i\int_{(\tau^n_j, \tau^n_{j+1})}\langle v_i^{\ast}(s),v(\tau^n_{j+1})-v(\tau^n_j)\rangle
\,dA(s)
$$
$$
=2\sum_i\int_{[\tau^n_j, \tau^n_{j+1})}\langle v_i^{\ast}(s),v_n^{(2)}(s)-v_n^{(1)}(s)\rangle
\,dA(s). 
$$
Using $\Pi_k$, the orthogonal projection of $H$ onto 
the space spanned by $(e_j)_{j=1}^k\subset V$, we have 
$$
J^{(22)}_j=J^{(22)}_{jk}+\bar J^{(22)}_{jk}
$$
with 
$$
J^{(22)}_{jk}:=2\left(\sum_i
\int_{(\tau^n_j, \tau^n_{j+1})} v_i^{\ast}(s)\,dA(s), \Pi_k(h(\tau^n_{j+1})-h(\tau^n_j)) \right), 
$$
$$
\bar J^{(22)}_{jk}:=2\left(\sum_i
\int_{(\tau^n_j, \tau^n_{j+1})} v_i^{\ast}(s)\,dA(s), (I-\Pi_k)(h(\tau^n_{j+1})-h(\tau^n_j)) \right).  
$$
Notice that 
$$
J^{(22)}_{jk}=2\sum_i
\int_{(\tau^n_j, \tau^n_{j+1})} 
\langle v_i^{\ast}(s),\Pi_k(h(\tau^n_{j+1})-h(\tau^n_j)) \rangle\,dA(s),
$$
$$
=2\sum_i
\int_{[\tau^n_j, \tau^n_{j+1})} 
\langle v_i^{\ast}(s),w^{(2)}_{kn}(s)-w^{(1)}_{kn}(s)\rangle\,dA(s),  
$$
and 
$$
\bar J^{(22)}_{jk}\leq J_j^{(3)}+\left|(I-\Pi_k)(h(\tau^n_{j+1})-h(\tau^n_j)) \right|^2. 
$$
Similarly, taking into account $\tilde v(0)=h(0)$, for  $J^{(2)}_0$ we have 
$$
J^{(2)}_0=2\left(\sum_i
\int_{(0, \tau^n_{1})} v_i^{\ast}(s)\,dA(s), \tilde v(\tau^n_{1})-h(\tau^n_{1}) \right)
=J^{(21)}_0-J^{(22)}_0,
$$
where 
$$
J^{(21)}_0:=2\sum_i\int_{(0, \tau^n_{1})}\langle v_i^{\ast}(s), \tilde v(\tau^n_{1})\rangle\,dA(s)
$$
$$
=2\sum_i\int_{(0, \tau^n_{1})}\langle v_i^{\ast}(s), v^{(2)}(s)-v^{(1)}(s)\rangle\,dA(s), 
$$
and 
$$
J^{(22)}_0:=2\left(\sum_i
\int_{(0, \tau^n_{1})} v_i^{\ast}(s)\,dA(s), h(\tau^n_{1})\right)=J^{(22)}_{0k}+\bar J^{(22)}_{0k}
$$
with 
$$
J^{(22)}_{0k}:=\sum_i
\int_{(0, \tau^n_{1})}\langle v_i^{\ast}(s),w^{(2)}(s)-w^{(1)}(s)\rangle\,dA(s), 
$$
$$
\bar J^{(22)}_{0k}:=2\left(\sum_i
\int_{(0, \tau^n_{1})} v_i^{\ast}(s)\,dA(s), (I-\Pi_k)h(\tau^n_1) \right) 
$$
$$
\leq J_0^{(3)}+\left|(I-\Pi_k)h(\tau^n_{1}) \right|^2. 
$$
Thus from \eqref{Kn} we get 
\begin{equation*}
\begin{split}
K_n(t) 
 \leq &\sum_{0\leq \tau^n_{j+1} \leq t}  
|\sum_iv_i^{\ast}(\tau^n_{j+1})|^2 |\Delta A(\tau^n_{j+1})|^2 \\
&+2\sum_i\int_{(0, t)} \langle v_i^{\ast}(s), v^{(2)}_n(s)-v^{(1)}_n(s) \rangle\, dA(s)
\\
&-2 
 \sum_i\int_{(0,t)} \langle v_i^{\ast}(s), w^{(2)}_{nk}(s) - w^{(1)}_{nk}(s)\rangle\,dA(s)
 +\xi_{nk}(t)  
\end{split}	
\end{equation*}
with 
$$
\xi_{nk}(t):=\sum_{j=1}^{N(n)} 
\left|(I-\Pi_k)(h(\tau^n_{j+1}\wedge t) - h(\tau^n_j\wedge t)) \right|^2	
+ \left|(I-\Pi_k)h(\tau^n_1\wedge t)\right|^2
$$
for every $n,k\in \N$.
As $n\to \infty$ we see that 
\begin{equation*}
 \sum_{0\leq \tau^n_{j+1} \leq t}  
|\sum_iv_i^{\ast}(\tau^n_{j+1})|^2 |\Delta A(\tau^n_{j+1})|^2  \to \sum_{0<s
\leq t} |v^{\ast}(s)|^2|\Delta A(s)|^2, 	
\end{equation*}
where we use the notation $v^{\ast}(s)=\sum_iv^{\ast}_i(s)$. 
By H\"older's inequality, taking into account $r(t)\leq1$, we have  
$$
\varlimsup_{n\to\infty}\int_{(0, t)} |\langle v_i^{\ast}(s), v^{(2)}_n(s)-v^{(1)}_n(s) \rangle|\, dA(s)
$$
$$
\leq \varlimsup_{n\to\infty}
\left(\int_{(0, t)} \|v^{(2)}_n(s)-v^{(1)}_n(s)\|_{V_i}^{p_i}\, dA(s)\right)^{\tfrac{1}{p_i}}=0
$$
and similarly, 
\begin{equation*}
\varlimsup_{n\to\infty}\int_{(0,t)} |\langle v_i^{\ast}(s), w^{(2)}_{nk}(s) - w^{(1)}_{nk}(s)\rangle|\,dA(s)=0
\end{equation*}
for all integers $k\geq1$ and $i=1,2,\ldots,m$. 
Thus
\begin{equation}                                                                   \label{K}
K(t)=\varliminf_{n\to\infty}K_n(t) \leq \sum_{s\leq t} |{v}^{\ast}(s)|^2|\Delta A(s)|^2 
+ \xi_{k}(t) 	
\end{equation}
for every $k\in N$, where  
\begin{equation*}
\xi_k(t): = \varliminf_{n\to \infty} \left(\sum_{j=1}^{N(n)} 
\left|(I-\Pi_k)(h(\tau^n_{j+1}\wedge t) - h(\tau^n_j\wedge t)) \right|^2	
+ \left|(I-\Pi_k)h(\tau^n_1\wedge t)\right|^2 \right).		
\end{equation*}
Note that by Fatou's lemma and the martingale property of $h$ 
\begin{equation*}
\begin{split}
& \E \xi_k(t) \leq \E|(I-\Pi_k)h(1)|^2  \to 0
\,\,\, \textrm{as}\,\,\, k\to \infty.
\end{split}	
\end{equation*}
Note also that for each $\omega \in \Omega$ and  
$t\in (0,\infty)$ we have $\xi_k \geq \xi_{k+1}$ and $\xi_k \geq 0$. 
Thus there exists a set $\Omega''''\subset\Omega$ with 
$P(\Omega'''')=1$ such that for every $t\in [0,\infty)$ and 
$\omega\in\Omega''''$ we have $\xi_k(t) \to 0$. 
Letting here $k\to\infty$ in \eqref{K} we obtain 
\begin{equation*}
K(t)\leq \int_{(0,t]}|{v}^{\ast}(s)|^2 \Delta A(s)\,dA(s), 	
\end{equation*}
which together with \eqref{Knbelow} gives 
\begin{equation*}
K(t)= \int_{(0,t]}|v^{\ast}(s)|^2 \Delta A(s)\,dA(s) 	
\end{equation*}
for $\omega \in \tilde{\Omega}:=\Omega'''\cap\Omega''''$ and $t\in I(\omega)$. 
\end{proof}

\begin{proposition}
For $\omega\in\tilde\Omega$ 
\begin{equation}                                                       \label{formula}
\begin{split}
& |\tilde v(t)|^2 = |h(0)|^2	
+ 2\sum_{i=1}^m \int_{(0,t]} \langle v_i^{\ast}(s), v(s) \rangle\, dA(s) \\
& + 2 \int_{(0,t]} (\tilde v(s-), dh(s)) 
 - \int_{(0,t]} \left|\sum_{i=1}^m v_i^{\ast}(s)\right|^2 \Delta A(s)\,dA(s)+[h]_t
\end{split}
\end{equation}
for $t\in[0,\tau(\omega))$. 
\end{proposition} 
\begin{proof}
Let $\omega \in \tilde{\Omega}$ be fixed and let $t\in[0,\tau(\omega))$. 
To ease notation we use $n\to\infty$ in place of the subsequence 
$n'\to\infty$ defined in the previous proposition. 
If $t\in I(\omega)$, then by virtue of the previous proposition 
taking $n\to \infty$ in~\eqref{eq:towards_ito_5} 
we obtain
\begin{align*}
|\tilde v(t)|^2 
= &|h(t)|^2
+ 2\sum_{i=1}^m \int_{(0,t]} \langle v_i^{\ast}(s), v(s) \rangle\, dA(s) 
+ 2 \int_{(0,t]} (\tilde v(s-)-h(-), dh(s))  \\
& - \int_{(0,t]} \left|\sum_{i=1}^m v_i^{\ast}(s)\right|^2 \Delta A(s)\,dA(s). 
\end{align*}
Hence using the It\^o formula for Hilbert space valued processes
\begin{equation}                                          \label{eq:hilbert_ito_fla}
|h(t)|^2 = |h(0)|^2 + 2 \int_{(0,t]} (h(s-),dh(s)) + [h]_t, 	
\end{equation}
we get \eqref{formula} for $t\in I(\omega)$. 
If $t\in \bar{I}^l(\omega)\setminus I(\omega)$, then 
for sufficiently large $n$ there is $j=j(n)$ such that 
$t_n:=\tau^{n}_j(\omega)\in I(\omega)$ and $t_n\nearrow t$ 
for $n\to \infty$. 
Using the algebraic relationship
\begin{equation}                                        \label{algebra}
|\tilde v(s)-\tilde v(r)|^2
=|\tilde v(s)|^2-|\tilde v(r)|^2-2\left(\tilde v(r),\tilde v(s)-\tilde v(r)\right), 
\end{equation}
with $s:=t_n$, $r:=t_l$, 
and since~\eqref{formula} holds for every $t\in I(\omega)$, 
we get 
\begin{equation*}
\begin{split}
& |\tilde v(t_n) - \tilde v(t_l)|^2 
= 2\sum_{i=1}^m \int_{(t_l,t_n]} \langle v_i^{\ast}(s),v(s) \rangle\, dA(s) 
 + 2 \int_{(t_l,t_n]} (\tilde v(s-), dh(s)) \\
& - \int_{(t_l,t_n]} 
\left|
\sum_{i=1}^m v_i^{\ast}(s)
\right|^2 
\Delta A(s)\, dA(s) + [h]_{t_n} - [h]_{t_l} 
- 2(\tilde v(t_l),\tilde v(t_n)-\tilde v(t_l)) 
\end{split}
\end{equation*}
for $n>l$. Moreover
\begin{equation*}
\begin{split}
& 2(\tilde v(t_l),\tilde v(t_n)-\tilde v(t_l)) \\
& = 2\sum_{i=1}^m \int_{(t_l,t_n]} \langle v_i^{\ast}(s),v(t_l) \rangle dA(s) 
+ 2(\tilde v(t_l),h(t_n)-h(t_l)). 		
\end{split}
\end{equation*}
Hence by~\eqref{eq:hilbert_ito_fla} 
\begin{equation*}
\begin{split}
|\tilde v(t_n) - \tilde v(t_l)|^2 
=  & 2\sum_{i=1}^m \int_{(t_l,t_n]} \langle v_i^{\ast}(s),v(s) - v(t_l)\rangle dA(s)\\ 
& + 2 \int_{(t_l,t_n]} (\tilde v(s-)-h(s-)-(\tilde v(t_l)-h(t_l)), dh(s)) \\
& - \int_{(t_l,t_n]} \left|\sum_{i=1}^m v_i^{\ast}(s)\right|^2 \Delta A(s) dA(s) 
+ |h(t_n) - h(t_l)|^2\\
& =: 2I^1_{ln} + 2I^2_{ln} - I^3_{ln} +I^4_{ln}.
\end{split}
\end{equation*}
Since $h$ is  cadlag we have 
\begin{equation*}
\lim_{l\to \infty} \sup_{n>l} I^4_{ln} 
= \lim_{l\to \infty} \sup_{n>l} 	|h(t_n) - h(t_l)|^2 = 0.
\end{equation*}
By the previous proposition we get
\begin{equation*}
\begin{split}
& \lim_{l\to \infty} \sup_{n>l} |I^2_{ln}| 
= \lim_{l\to \infty} \sup_{n>l} 
\left|\int_{(t_l,t_n]} (\tilde v(s-)-h(s-)-(\tilde v_l(s)-h_l(s)), dh(s))\right|\\
& \leq 2\lim_{l\to \infty} \sup_{t\in (0,\infty)}
\left|\int_{(0,t]} (\tilde v(s-)-h(s-)-(\tilde v_l(s)-h_l(s)), dh(s))\right| = 0, 
\end{split}	
\end{equation*}
and 
\begin{equation*}
\begin{split}
& \lim_{l\to \infty} \sup_{n>l} |I^1_{ln}| 
\leq \lim_{l\to \infty}
\sum_{i=1}^m \int_{(0,\infty)} \|v_i^{\ast}(s)\|_{V_i^*}\|v(s)-v^{(1)}_l(s)\|_{V_i}\,dA(s)
= 0,  
\end{split}	
\end{equation*}
via $r(t)\leq1$  and 
H\"older's inequality. 
Thus
\begin{equation*}
\lim_{l\to \infty} \sup_{n>l} |\tilde v(t_n)-\tilde v(t_l)|^2 = 0, 	
\end{equation*}
and so the sequence $(\tilde v(t_n))_{n\in \N}$ converges strongly to some $\xi$ in $H$.
Moreover since $\tilde v$ is weakly cadlag and $t_n \nearrow t$, we conclude
that $\xi=\tilde v(t-)$.
Hence using  \eqref{formula} with $t_n$ in place of $t$, 
and letting $n\to\infty$ we obtain 
\begin{equation*}                                                    
\begin{split}
& |\tilde v(t-)|^2 = |h(0)|^2	
+ 2\sum_{i=1}^m \int_{(0,t)} \langle v_i^{\ast}(s), v(s) \rangle\, dA(s) \\
& + 2 \int_{(0,t)} (\tilde v(s-), dh(s)) 
 - \int_{(0,t)} \left|\sum_{i=1}^m v_i^{\ast}(s)\right|^2 \Delta A(s)\,dA(s)+[h]_{t-}
\end{split}
\end{equation*}
for $t\in I(\omega)\setminus\bar I^{l}(\omega)$, and so for this $t$ we get also 
\eqref{formula} by taking into account that $\Delta A(t)=0$. 
If $t \in (0,\tau(\omega)) \setminus \bar{I}^l(\omega)$, then there is 
$t'\in \{0\} \cup \bar{I}^l(\omega)$ 
such that $t'<t$ and $(t',t] \cap I(\omega) = \emptyset$. Thus 
$dA(s) = 0$ for $s\in (t',t]$, and so $\tilde v(s)-\tilde v(t') = h(s)-h(t')$.
Hence applying~\eqref{formula} with $t:=t'$, 
and the formula
$$
|\tilde v(t)|^2-|\tilde v(t')|^2=2(\tilde v(t'),\tilde v(t)-\tilde v(t'))+|\tilde v(t)-\tilde v(t')|^2
$$
together with the It\^o formula for Hilbert space valued martingales, 
\begin{equation*}
|h(t) - h(t')|^2 = 2\int_{(t',t]} (h(s-)-h(t'),dh(s)	) + [h]_t - [h]_{t'}, 
\end{equation*}
we obtain \eqref{formula} for the $t$ under consideration. 
\end{proof}
Now we can finish the proof of Theorem~\ref{thm:1} 
by noting that by the above proposition $|\tilde v(t)|^2$ is a cadlag process, and 
since  by Proposition~\ref{propn:2} the process $\tilde v$ is $H$-valued and weakly cadlag, 
it follows by identity \eqref{algebra}
that $\tilde v$ is an $H$-valued cadlag process.

\paragraph{\bf Acknowledgements}
The authors are sincerely grateful to the anonymous referees. 
Their corrections and valuable suggestions helped improve the presentation of the paper. 

\paragraph{\bf Open Access} This article is distributed under the terms of the Creative Commons Attribution 4.0 International License (\url{http://creativecommons.org/licenses/by/4.0/}), which permits unrestricted use, distribution, and reproduction in any medium, provided you give appropriate credit to the original author(s) and the source, provide a link to the Creative Commons license, and indicate if changes were made.

\end{document}